\documentclass[11pt]{elsarticle}

\usepackage{lineno,hyperref}
\usepackage{epstopdf}
\usepackage{subfigure}
\usepackage[margin=1.2in]{geometry}
\usepackage{slashbox}
\usepackage{graphicx}
\usepackage{tikz}
\usetikzlibrary{calc}
\usetikzlibrary{decorations.markings}
\usepackage{spverbatim}
\usepackage{amssymb}
\usepackage{array}
\newcolumntype{P}[1]{>{\centering\arraybackslash}p{#1}}
\usepackage{float}
\usepackage{multicol, multirow}
\usepackage{amsmath, amsthm}
\usepackage{commath}
\usepackage{enumerate}
\usepackage{color}
\usepackage{bm}
\usepackage{hyperref}
\usepackage{url}
\usepackage{mathptmx} 
\usepackage{mathtools}     
\usepackage{latexsym}
\theoremstyle{plain}
\newtheorem{theorem}{Theorem}[section]

\newtheorem{proposition}{Proposition}[section]
\newcounter{casenum}

\theoremstyle{definition}
\newtheorem{definition}{Definition}[section]

\theoremstyle{remark}
\newtheorem{remark}{Remark}

\def\tsc#1{\csdef{#1}{\textsc{\lowercase{#1}}\xspace}}
\tsc{WGM}
\tsc{QE}
\modulolinenumbers[5]

\journal{Journal of \LaTeX\ Templates}

\biboptions{authoryear}

\begin{document}
	\nolinenumbers
	\begin{frontmatter}
		
		\title{Geometric Infinitely Divisible Autoregressive Models}
		
		\author[mymainaddress]{Monika S. Dhull}
		
		\author[mymainaddress]{Arun Kumar \corref{mycorrespondingauthor}}
		\cortext[mycorrespondingauthor]{Corresponding author}
		\ead{arun.kumar@iitrpr.ac.in}
		
		
		\address[mymainaddress]{Department of Mathematics, Indian Institute of Technology Ropar, Punjab-140001, India.}
		
		
		
		
		
		\begin{abstract}
			In this article, we discuss some geometric infinitely divisible (gid) random variables using the Laplace exponents which are Bernstein functions and study their properties. The distributional properties and limiting behavior of the probability densities of these gid random variables at $0^+$ are studied. The autoregressive (AR) models with gid marginals are introduced. Further, the first order AR process is generalised to $k^{th}$ order AR process. We also provide the parameter estimation method based on conditional least square and method of moments for the introduced AR($1$) processes.
		\end{abstract}
		
		
	\end{frontmatter}

	\section{Introduction}\label{intro}
	The classical autoregressive (AR) processes are the most popular time series models, since these models are easy to understand, interpretable, and have applications in different fields. In AR processes the next term is a linear combination of previous terms and an innovation term. The distribution of innovation terms plays an important role to capture the extreme events. The classical AR model assumes that the innovation terms has Gaussian distribution, which further leads to the marginals to be Gaussian. However, many real life time series marginals exhibit heavy-tailed behaviour or semi-heavy tailed behaviour. In literature various AR models are introduced with non-Gaussian innovations and marginals. The data with non-negative observations, binary outcomes, series of counts, proportions are some examples of non-Gaussian real life time series data (see e.g. \cite{GAG1993}). The first order AR models with different marginal distributions are very well studied in literature.  \cite{Gav1980} considered the AR(1) model with marginals from gamma distribution. The authors discussed that the innovation terms have distribution same as the distribution of a non-negative random variable which is exponential for positive values and has a point mass at $0$. A second order autoregressive model is considered by \cite{Dew1985}, where Laplace distributed marginals are assumed. \cite{Abba1999} considered inverse Gaussian autoregressive model where the marginals are inverse Gaussian distributed. The authors shows that after fixing a parameter of the inverse Gaussian to $0$ the innovation term is also inverse Gaussian distributed. Recently, \cite{niha2022} studied the AR($1$) model with one-sided tempered stable marginals and innovations and demonstrated that the model fits very well on real world data. The first order AR processes are also studied by \cite{Jose2004, Jose2006} with geometric $\alpha$-Laplace and geometric Pakes generalised Linnik marginals. For a literature survey on unified view of AR(1) models see \cite{Grunwald1996}. \\
	In this article, we use the concept of geometric infinitely divisible distributions to define a new class of generalised AR($1$) processes.  We first define a class of gid random variables with Laplace transform of the form $\frac{1}{1+g(s)}$, where $g(s)$ is a Bernstein function which is the Laplace exponent of general subordinators. A large class of distributions e.g. the geometric $\alpha$-stable (or Mittag Leffler), geometric tempered stable, geometric gamma, geometric inverse Gaussian, one-sided geometric Laplace distributions and others come under this umbrella. We mainly focus on geometric tempered stable, geometric inverse Gaussian and geometric gamma distributions and corresponding AR models in this article. As per our knowledge these models and corresponding AR processes are not explored in literature. First the properties of gid marginals are studied in section \ref{section2}. Section \ref{GBAR_model}
	describes the AR$(1)$ model with gid marginals. We also obtain the distribution of innovation terms which are from the class of gid distributions except scale change. We then evaluate the pdf of innovation terms in integral form for the geometric tempered stable, geometric gamma, and geometric inverse Gaussian. Further, we generalise the result obtained in this section to the $k^{th}$ order autoregressive processes. In section \ref{section4} we use the conditional least square (CLS) and method of moments (MOM) to estimate the parameters of the AR($1$) model defined in \ref{ARmodel} and provide the simulation study for the same.
	Section \ref{conclusion} concludes the work.   
	
	
	\section{Geometric infinite divisible distributions}\label{section2}
	In this section, we introduce gid distributions. Infinitely divisible random variables (or distributions) with positive support play a crucial role in the study of subordinators \cite{applebaum}. The subordinators are real-valued non-decreasing stochastic processes having independent and stationary increments. The properties of subordinators are well studied in literature (for e.g. see, \cite{applebaum, arun2018, kumar_vellai}). Bernstein functions play a crucial role in the theory of non-decreasing L\'evy processes or subordinators. A function  $\phi: (0,\infty) \rightarrow (0, \infty)$ is called a Bernstein function if it admits the following representation
	$$
	\phi(s) = a + bs + \int_{0}^{\infty}(1-e^{-sx})\nu(dx), 
	$$
	where $a\geq 0$, $b\geq 0$ and $\nu$ is called the L\'evy measure which is defined on $(0, \infty)$. Alternatively, the function $\phi: (0,\infty) \rightarrow (0, \infty)$ is a Bernstein function if $\phi$ is of class $\mathbb{C^{\infty}}$ and $\phi(s) \geq 0$, for all $s\geq 0$ and $(-1)^{n-1}\phi^{(n)}(s) \geq 0,$ for all $n\in \mathbb{N}$ and $s>0$ (\cite{book1}). For a subordinator $S(t)$, it follows that $\mathbb{E}[e^{-sS(t)}] = e^{-t\phi(s)}$, where $\phi(s)$ is also called the Laplace exponent.  Next we define the gid' marginals.\\
	\noindent Consider a positive random variable $Y$ with Laplace transform $f(s)$ that is $\mathbb{E}[e^{-sY}] = f(s).$ Let $X$ be a positive infinitely divisible random variable with Laplace transform $\mathbb{E}[e^{-sX}] = e^{-g(s)}$, where $g(s)$ is a Bernstein function (\cite{book1}). Now we can write the Laplace transform of $X$ as 
	\begin{equation}\label{gid}
		\phi_X(s) = e^{-g(s)} =  e^{1-g(s)-1} = e^{1-\frac{1}{(1+g(s))^{-1}}} = e^{1- \frac{1}{f(s)}},\; s>0.
	\end{equation}
	\noindent It is known from \cite{Klebanov1985} that a distribution with Laplace transform $f(s)$ is geometrically infinitely divisible if and only if the distribution with Laplace transform $e^{1-\frac{1}{f(s)}}$ is infinitely divisible. One can define a class of geometrically infinitely divisible random variables by choosing $g(s)$ as the Laplace exponent of general subordinators and we call these random variables as the gid based random variables. We introduce a class of random variables $Y$ on ($0,\infty$) with Laplace transform of the form $f(s) = \frac{1}{1+g(s)}$, where $g(s)$ is the Laplace exponent of positive infinitely divisible random variables and study the relevant properties for time series autoregressive models.
	We extend the concept of geometric $\alpha-$Laplace (\cite{Jose2004}), geometric Pakes generalised Linnik marginals (\cite{Jose2006}) based autoregressive models to general class of autoregressive models with gid marginals. 
	
	\begin{definition}\label{def2}
		The random variable $Y$ on ($0,\infty$) is said to have gid marginals if its Laplace transform is given by $f(s) = \frac{1}{1+g(s)}$, where $g(s)$ is a Bernstein function and the Laplace exponent of positive infinitely divisible random variable. 		
	\end{definition}
	
	\noindent In particular, we consider the different functions for $g(s)$ and define the Laplace transform of corresponding gid marginals:
	\begin{enumerate}[(a)]
		\item Geometric tempered stable:
		The Laplace transform of tempered stable random variable is $e^{-g(s)}$ with $g(s) = (s+\lambda)^{\beta}-\lambda^{\beta}$, then $f(s) = \frac{1}{1+(s+\lambda)^{\beta}-\lambda^{\beta}}, \, \lambda>0 \text{ and }\beta\in(0,1).$ 
		
		\item Geometric gamma: 
		The Laplace transform of gamma random variable is $e^{-g(s)}$ with $g(s) = \alpha\log(1+\frac{s}{\beta}),$ then $f(s) = \frac{1}{1+\alpha \log\left(\frac{s+\beta}{\beta}\right)}, \, \alpha>0 \text{ and } \beta>0.$
		
		\item Geometric inverse Gaussian:
		The Laplace transform of inverse Gaussian random variable is $e^{-g(s)}$ with $g(s) = \delta\gamma\Bigl\{\sqrt{1+\frac{2s}{\gamma^2}}-1\Bigr\}$, then $f(s) = \frac{1}{1+\delta\gamma\Bigl\{\sqrt{1+\frac{2s}{\gamma^2}}-1\Bigr\}}, \, \delta>0 \text{ and } \gamma>0.$
	\end{enumerate}
	
	\begin{proposition}
		Consider the identically and independently distributed random variables $X_1, X_2, \hdots,$ with gid marginals defined in Def. \ref{def2}. Also consider $N(\theta)$ be a geometric random variable with mean $\frac{1}{\theta}$ and $P[N(\theta)=r] = \theta(1-\theta)^{r-1}$, $r = 1,2,\cdots, \, 0<\theta<1.$ Then $Y = \sum_{i=1}^{N(\theta)} X_{i}$ also have gid marginals with a scale change.  
	\end{proposition}
	
	\begin{proof} 
		Let the Laplace transform of each $X_i$ for $i=1,2, \hdots$ is denoted by $\phi_X(s)$.	Then, the Laplace transform of random variable $Y$ can be written as,
		\begin{align*}
			\phi_{Y}(s) &= \sum_{r=1}^{\infty}[\phi_{X}(s)]^{r} \theta(1-\theta)^{r-1}\, =\sum_{r=1}^{\infty}\Big[\frac{1}{1+g(s)}\Big]^{r} \theta(1-\theta)^{r-1}\\
			&=\frac{\frac{\theta}{1+g(s)}}{1-\big(\frac{1-\theta}{1+g(s)}\big)}\, =\frac{\theta}{\theta + g(s)}\,=\frac{1}{1+\frac{g(s)}{\theta}}.
		\end{align*}
		Hence, $Y$ is also a gid random variable.
	\end{proof}
	\noindent Next, we study the limiting behavior of the densities of gid random variables near $0^+$ using Laplace transform. Since the pdf of gid don't have closed form i.e. they have integral representations, therefore we apply the Tauberian theorem defined in \cite{Feller1971} (pp. 446, Theorem 4). In the following theorem we apply the Tauberian theorem on the geometric tempered stable and geometric inverse Gaussian cases.
	
	\begin{theorem}\label{Theorem2.1}
		Consider the following Laplace transform $W(s)$ for their corresponding density function $f(x)$, then we obtain the asymptotic behavior of $f(x)$ at $0^{+}$.
		\begin{enumerate}[(a)]
			\item For geometric tempered stable, $W(s)=\frac{1}{1+\theta((s+\lambda)^\beta-\lambda^\beta)}$, then as $x\to 0^+$ density $f(x)\sim \frac{x^{\beta-1}}{\theta\Gamma(\beta)}$, where $\lambda>0$ and $0 < \beta < 1$.
			
			
			\item For geometric inverse Gaussian, $W(s)=\frac{1}{1+\theta\delta\{\sqrt{\gamma^2+2s}-\gamma\}}$, then as $x\to 0^{+}$ density $f(x)\sim \frac{x^{-1/2}}{\theta\delta\sqrt{2\pi}}$, where $\delta>0$, and $ \gamma>0$.
		\end{enumerate}
	\end{theorem}
	\begin{proof}
		\begin{enumerate}[(a)]
			\item The Laplace transform of geometric tempered stable is $W(s)=\frac{1}{1+\theta((s+\lambda)^\beta-\lambda^\beta)}$, where $\lambda>0,\,0 < \beta < 1.$ For large $s$, $W(s)\sim \frac{1}{\theta s^{\beta}}.$ We apply Tauberian theorem with $\rho = \beta$ and the slowly varying function $L(1/s)=\frac{1}{\theta}$. Thus, the density 
			$$f(x)\sim \frac{x^{\beta-1}}{\theta\Gamma(\beta)}, \text{ as } x\to 0^{+}.$$
			
			
			\item The Laplace transform of geometric inverse Gaussian is $W(s)=\frac{1}{1+\theta\delta\{\sqrt{\gamma^2+2s}-\gamma\}}$, where $\delta>0, \gamma>0$. Again for large $s$, $W(s)\sim \frac{1}{\theta \delta \sqrt{2s}}.$ The slowly varying function $L(1/s)=\frac{1}{\theta\delta \sqrt{2}}$ and $\rho=1/2.$ By Tauberian theorem, $f(x)\sim \frac{x^{-1/2}}{\theta\delta\sqrt{2\pi}},\text{ as } x\to 0^{+}$.
		\end{enumerate}   
	\end{proof}
	\noindent The density plots of the gid distributions specially, geometric tempered stable and geometric inverse Gaussian are also in accordance with the results obtained in above theorem. We use the method of Laplace transform applied in \cite{Ridout2009} to generate the data of length $1000$ in $R$ programming language. The density plots of gid with different parameter values as shown in Fig. \ref{fig2}.  
	\begin{figure}[H]
		\centering
		\subfigure[Geometric tempered stable]{\includegraphics[width=0.32\textwidth]{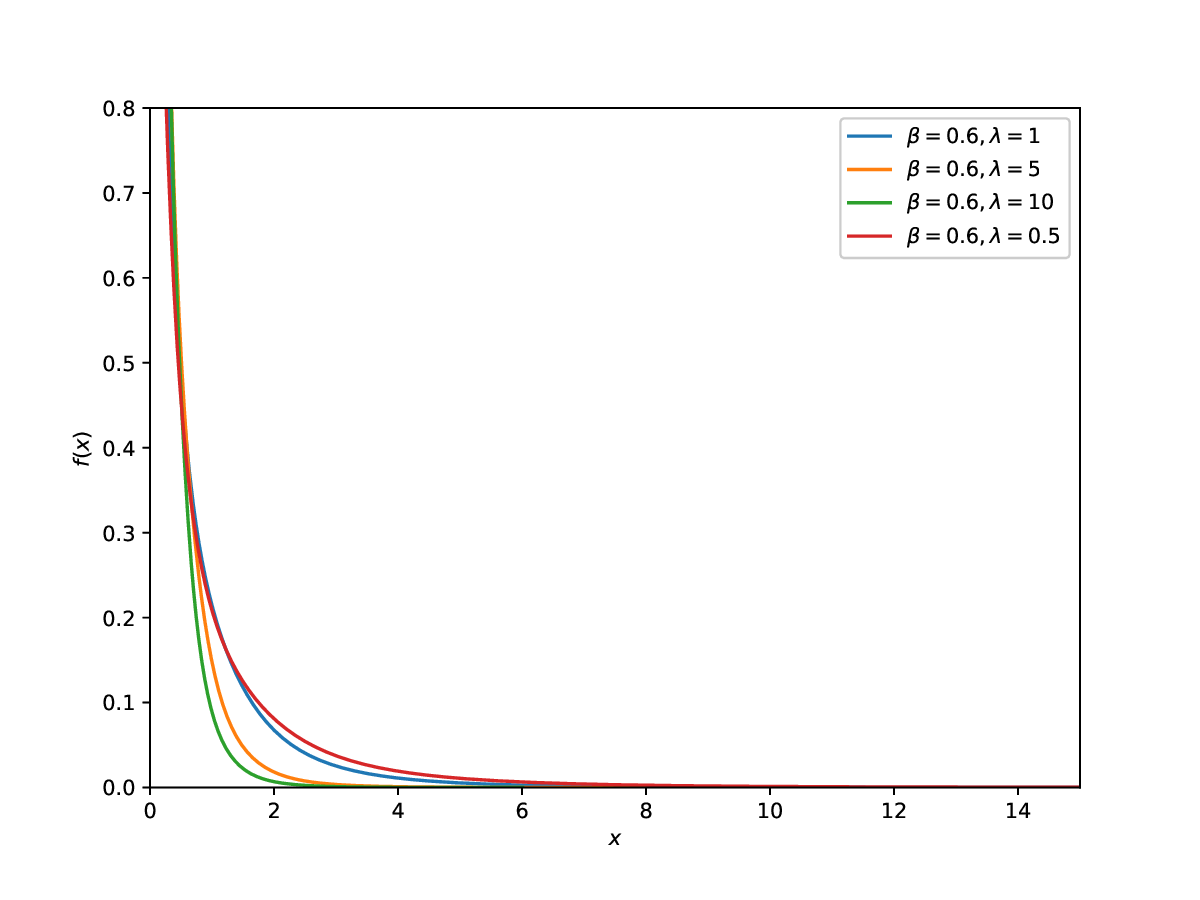}}\hfill
		\subfigure[Geometric gamma]{\includegraphics[width=0.32\textwidth]{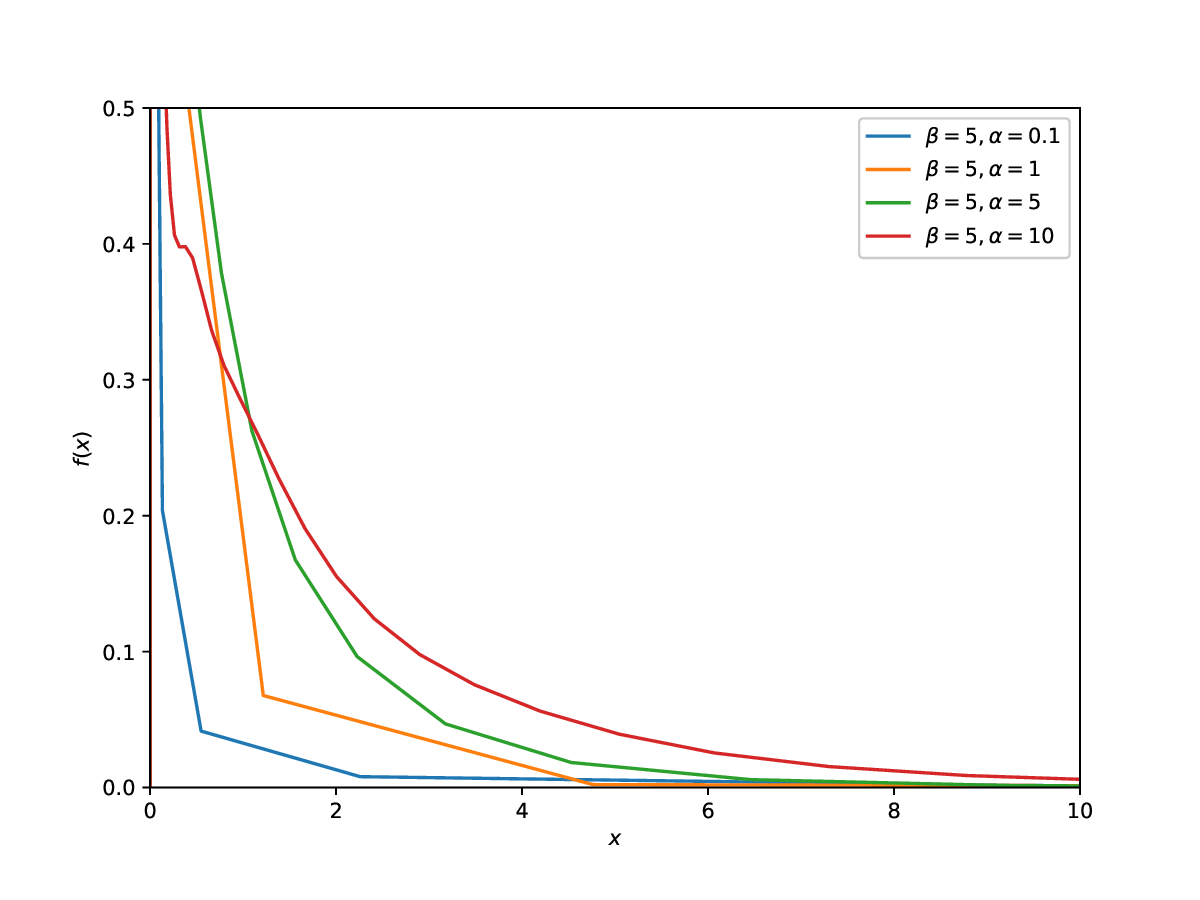}} \hfill
		\subfigure[Geometric inverse Gaussian]{\includegraphics[width=0.32\textwidth]{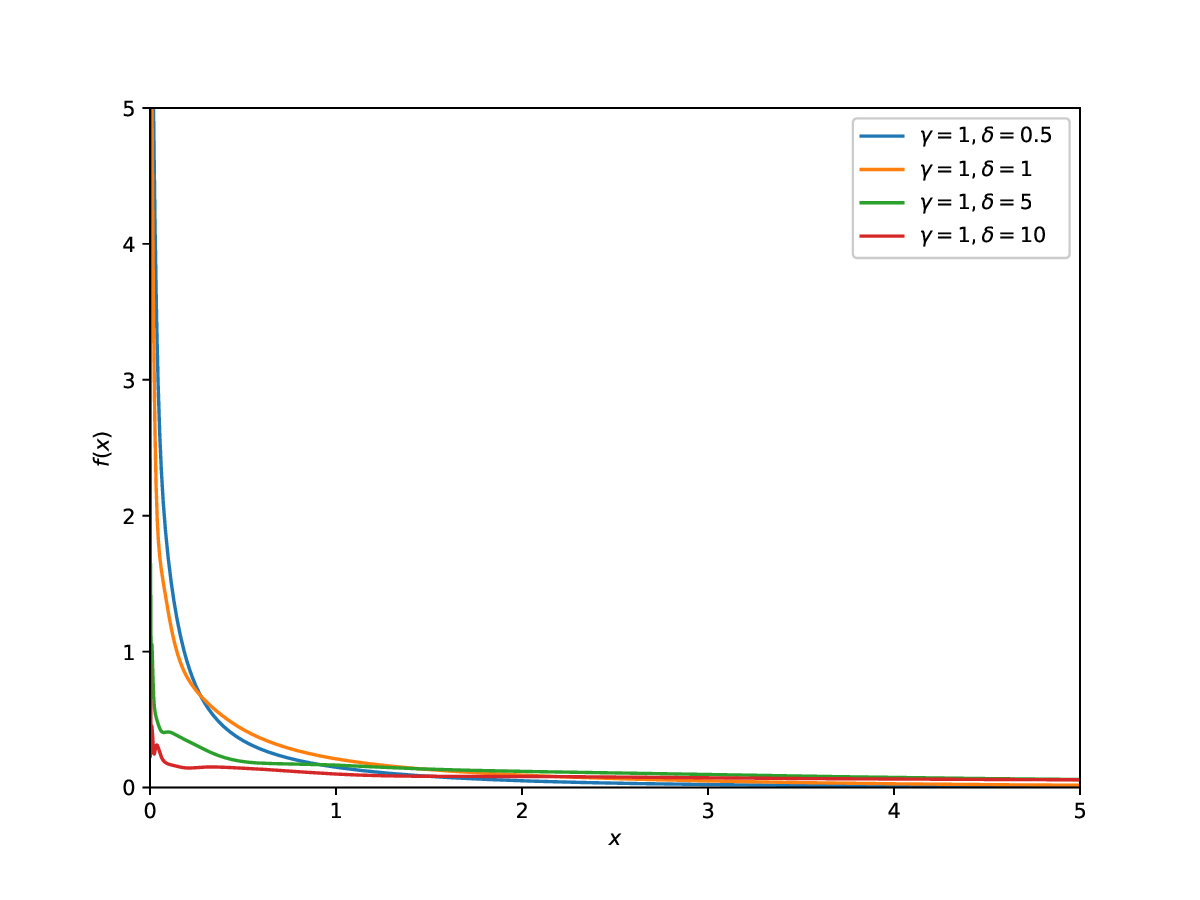}}
		\caption{The density plot of (a) geometric tempered stable with parameters $\beta=0.6$ and different $\lambda=1,5,10,0.5$ (b) geometric gamma with parameters $\beta=5$ and different $\alpha=0.1,1,5,10$ (c) geometric inverse Gaussian with parameters $\gamma=1$ and different $\delta=0.5,1,5,10$. }
		\label{fig2}
	\end{figure}
	
	\subsection{Mixture of gid random variables}
	We introduce a new random variable $M$ which is a mixture of two independent gid random variables say $Y_1$ and $Y_2$. We define $g(s) = cg_1(s) + (1-c)g_2(s)$, where $0<c<1$, $g_1$ and $g_2$ are Laplace exponent functions for $Y_1$ and $Y_2$ respectively.
	We compute $\phi_M(s)$ the Laplace for each case as follows:
	\begin{align*}
		e^{-g(s)} &= e^{-(c g_1(s)+(1-c)g_2(s))}= e^{1-1-(c g_1(s)+(1-c)g_2(s))} \\
		&= e^{1-\frac{1}{(1+cg_1(s) +(1-c)g_2(s))^{-1}}}
	\end{align*}
	The Laplace of gid mixture $M$ is,
	$$\phi_M(s) = \frac{1}{1+cg_1(s) +(1-c)g_2(s)}\\
	=\frac{1}{1 + g_2(s) + c(g_1(s)-g_2(s))},$$ for $0<c<1$.
	We substitute the values of functions $g_1(s)$ and $g_2(s)$ to obtain the Laplace transform of the following:
	\begin{enumerate}[(a)]
		\item Geometric mixture tempered stable: For $0<c<1,\, 0<\beta_1, \beta_2<1$ and $ \lambda_1, \lambda_2>0,$
		$$\phi_M(s) = \frac{1}{1 + (s+\lambda_2)^{\beta_2}-\lambda_2^{\beta_2} + c((s+\lambda_1)^{\beta_1}-\lambda_1^{\beta_1}-(s+\lambda_2)^{\beta_2}+\lambda_2^{\beta_2})}.$$
		
		\item Geometric mixture gamma: For $0<c<1,\, \beta_1, \beta_2>0$ and $ \alpha_1, \alpha_2>0,$
		$$\phi_M(s) = \frac{1}{1 + \alpha_2\log(\frac{s+\beta_2}{\beta_2})+ c\{\alpha_1\log(\frac{s+\beta_1}{\beta_1})-\alpha_2\log(\frac{s+\beta_2}{\beta_2})\}}.$$

		\item Geometric mixture inverse Gaussian: For $0<c<1,\, \delta_1, \delta_2>0$ and $ \gamma_1, \gamma_2>0,$
		$$\phi_M(s) = \frac{1}{1 + \delta_2\Bigl(\sqrt{2s+\gamma_2^2}-\gamma_2\Bigr)+ c\Bigl(\delta_1\Bigl(\sqrt{2s+\gamma_1^2}-\gamma_1\Bigr) -\delta_2\Bigl(\sqrt{2s+\gamma_2^2}-\gamma_2\Bigr)\Bigr)}.$$
	\end{enumerate}
	\begin{remark}
		One can study the asymptotic behavior of gid mixture as done in \ref{Theorem2.1}. Also, further these gid mixtures can be used to define autoregressive models as discussed in next section. 
	\end{remark}
	\section{Autoregressive model}\label{GBAR_model}
	
	\noindent In this section, we develop the AR$(1)$ process with  $\{Y_{n}\}$ as gid based marginals. Consider the following AR$(1)$ process:
	
	\begin{equation} \label{GBmodel}
		Y_{n} =
		\begin{cases}
			\epsilon_n, & \text{ with probability } \theta,\\
			Y_{n-1} + \epsilon_n, & \text{ with probability } 1-\theta,
		\end{cases} 
	\end{equation}
	where $0<\theta<1$.
	
	\begin{theorem}
		Consider a stationary AR$(1)$ process $\{Y_{n}\}$ as defined in Eq. \eqref{GBmodel}. Then $\{Y_{n}\}$ and $\{\epsilon_{n}\}$ are identically distributed except for a scale change if and only if $\{Y_{n}\}$ has gid marginals. 
	\end{theorem}
	
	\begin{proof} First assume that the marginal distribution of $\{Y_n\}$ is are gid. Then the Laplace transform of $\{Y_{n}\}$ is $\phi_{Y_n}(s)= \frac{1}{1+g(s)}$ and innovation terms $\{\epsilon_n\}$ is denoted by $\phi_{\epsilon_n}(s)$. We can write Eq. \eqref{GBmodel} in terms of Laplace transform as,
		$$\phi_{Y_n}(s) = \theta \phi_{\epsilon_n}(s) + (1-\theta)\phi_{Y_{n-1}}(s)\phi_{\epsilon_n}(s).$$
		Using the stationarity condition and the Laplace transform $\phi_{Y}(s) =\frac{1}{1+g(s)}$, we rewrite above equation as,
		\begin{align}
			\phi_{Y}(s) &= \theta \phi_{\epsilon}(s) + (1-\theta)\phi_{Y}(s)\phi_{\epsilon}(s)\,
			= \phi_{\epsilon}(s)\{\theta + (1-\theta)\phi_{Y}(s)\}\\
			\phi_{\epsilon}(s)&= \frac{\phi_{Y}(s)}{\theta + (1-\theta)\phi_{Y}(s)}\,
			= \frac{1}{1 + \theta g(s)}.\label{main_eq2}
		\end{align}
		
		\noindent Hence, we conclude that innovation terms \{$\epsilon_n$\} are gid with $\theta$ scale change.\\
		
		Conversely, assume that $\{Y_{n}\}$ and $\{\epsilon_{n}\}$ are identically distributed except for scale change say $\phi_{\epsilon}(s) = \phi_{Y}(as)$, where $a$ is a constant, $\phi_{Y}(s)$ and  $\phi_{\epsilon}(s)$ is Laplace transform of $\{Y_{n}\}$ and innovation terms $\{\epsilon_n\}$ respectively. Using $\phi_{\epsilon}(s) = \frac{1}{1+g(s)}$ and $\phi_{Y}(s)$, we get,
		\begin{align*}
			\phi_{Y_n}(s) &= \theta \phi_{\epsilon_n}(s) + (1-\theta)\phi_{Y_n}(s)\phi_{\epsilon_n}(s)\,
			= \phi_{\epsilon}(s)\{\theta + (1-\theta)\phi_{Y}(s)\}\\
			\phi_{Y}(s)	&= \frac{\theta \phi_{\epsilon}(s)}{1-\{1-\theta \phi_{\epsilon}(s)\}}\,= \frac{\theta}{1+g(s)-(1-\theta)}\\
			&=\frac{\theta}{\theta + g(s)}\,= \frac{1}{1+\frac{1}{\theta}g(s)}.		\end{align*}
		
		\noindent Take $a = \frac{1}{\theta}$, we get that $\{Y_n\}$ are also gid with scale change of $\frac{1}{\theta}$.\\
		Hence, we have the Laplace transform of $\{\epsilon_n\}$ as gid with scale change of $\theta$ and the marginals $\{{Y_n}\}$ will also be same as innovation terms, that is gid.
	\end{proof}
	
	\noindent We find the joint Laplace transform of $Y_n$ and $Y_{n-1}$ and check for time reversibility of the process in the following result.
	\begin{proposition}
		Consider the stationary AR$(1)$ process as defined in Eq. \eqref{GBmodel}. Then we can rewrite $Y_n$ as $Y_n = I_nY_{n-1} + \epsilon_n$, where $P[I_n=0] = \theta = 1-P[I_n=1], 0<\theta<1$. Then the $AR(1)$ process is not time reversible.
	\end{proposition}
	\begin{proof}
		The Laplace transform of joint variables $(Y_n, Y_{n-1})$ is calculated as follows:
		\begin{align*}
			\phi_{Y_{n-1}, Y_n}(s_1, s_2) &= \mathbb{E}[\exp(-s_1 Y_{n-1} - s_2Y_n)] \,= \mathbb{E}[\exp(-s_1 Y_{n-1} - s_2[I_nY_{n-1} + \epsilon_n]] \\
			&= \mathbb{E}[\exp(-(s_1 + s_2I_n) Y_{n-1} - s_2\epsilon_n)] \, = \mathbb{E}[\exp(-(s_1 + s_2I_n) Y_{n-1})] \phi_{\epsilon_n}(s_2)\\
			&=\frac{1}{1+\theta g(s_2)} \Big[\frac{\theta}{1+g(s_1)} + 
			\frac{1-\theta}{1+g(s_1 + s_2)}\Big].
		\end{align*}
		The joint Laplace transform of the process implies that it is not time reversible.
	\end{proof}
	\begin{figure}[H]
		\centering
		\subfigure[Geometric tempered stable]{\includegraphics[width=0.33\textwidth]{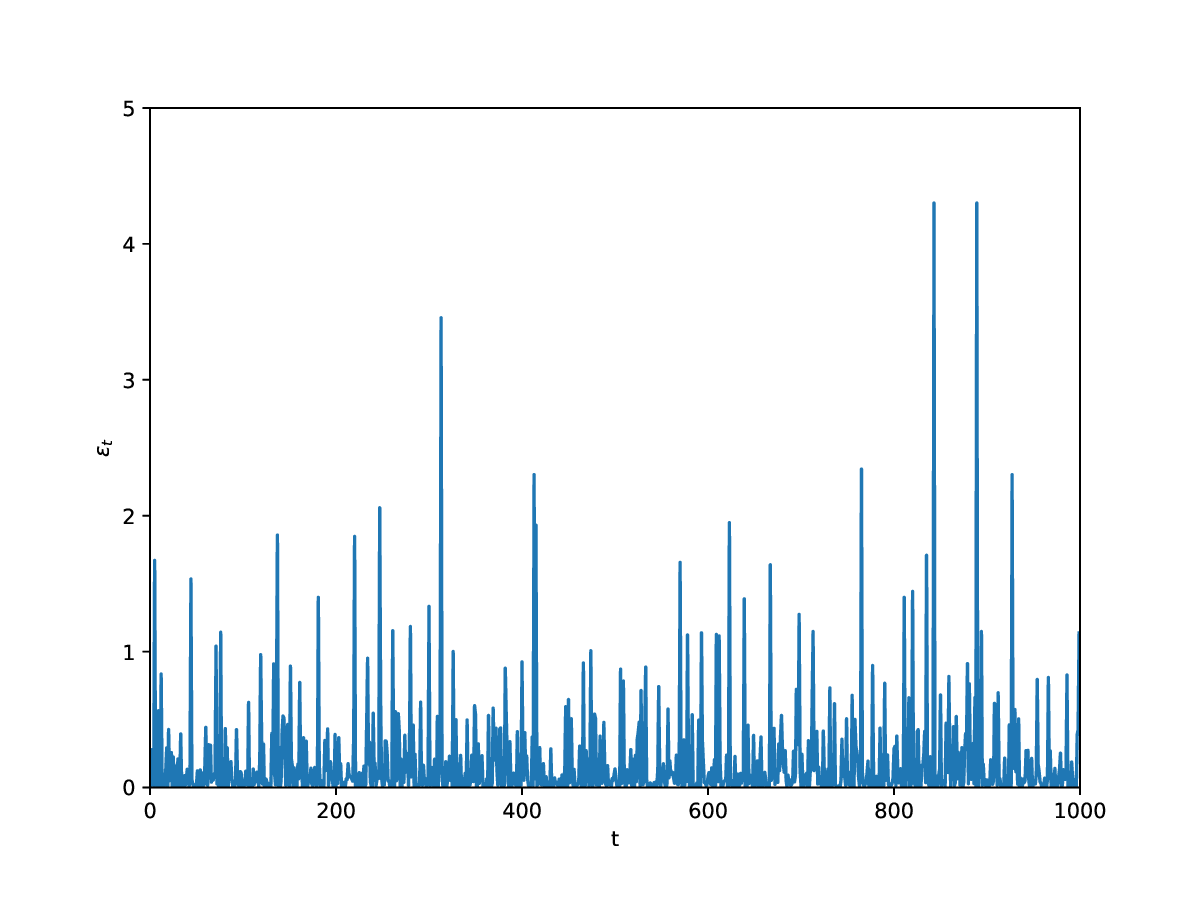}}\hfill
		\subfigure[Geometric gamma]{\includegraphics[width=0.33\textwidth]{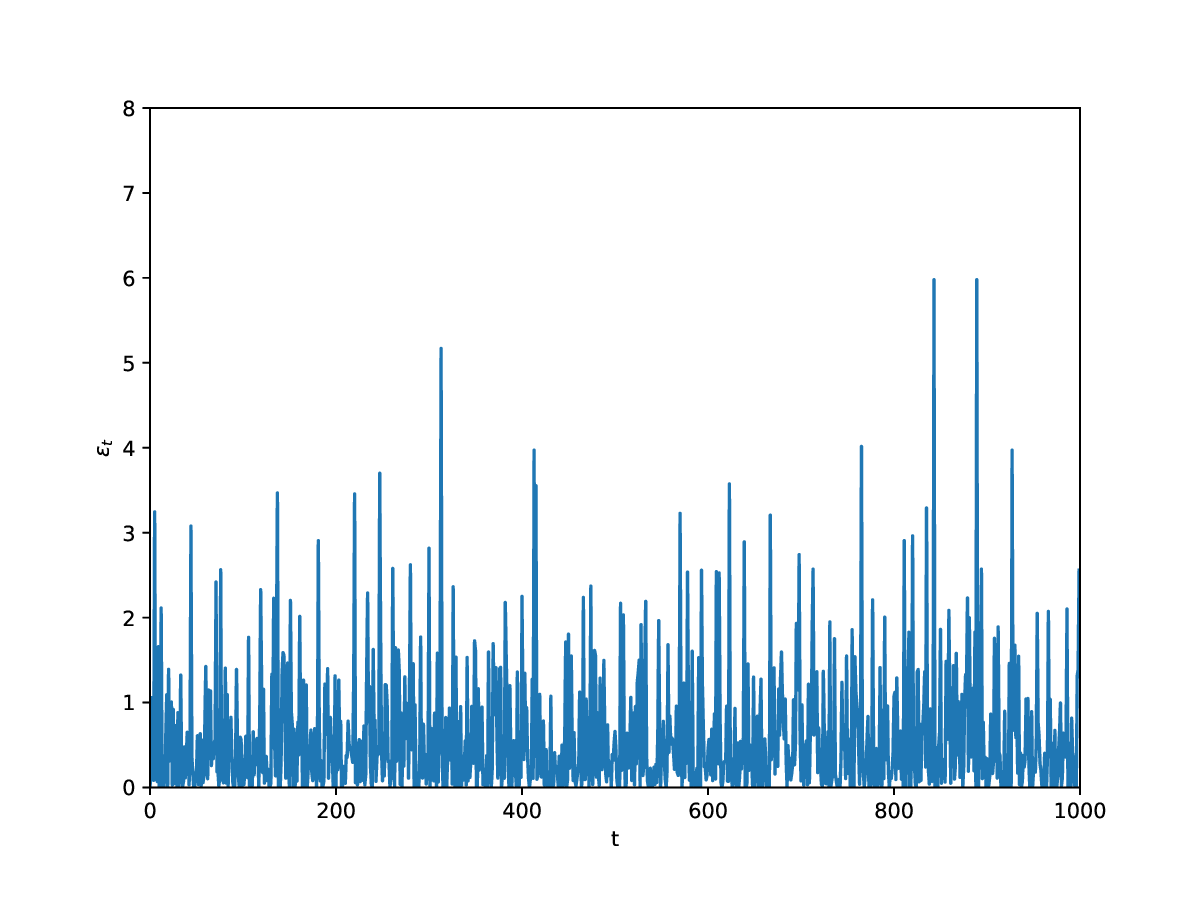}} \hfill
		\subfigure[Geometric inverse Gaussian]{\includegraphics[width=0.33\textwidth]{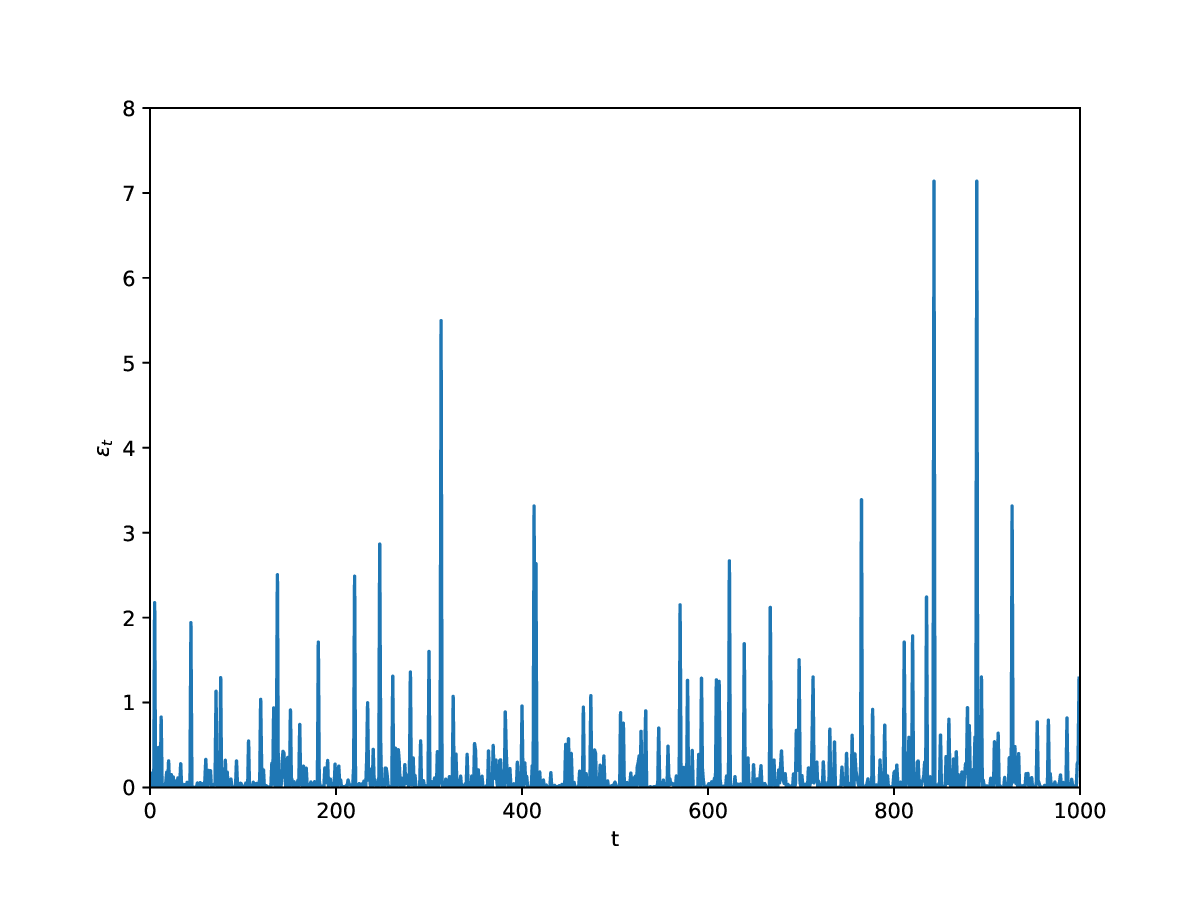}}
		\caption{The plot of innovation terms $\{\epsilon_t\}$ from AR($1$) model with $\theta =0.3$ for (a) geometric tempered stable with parameters $\beta=0.6$ and $\lambda=1$ (b) geometric gamma with parameters $\beta=5$ and $\alpha=10$ (c) geometric inverse Gaussian with parameters $\gamma=1$ and $\delta=0.5$. }
		\label{fig4}
	\end{figure}
	\noindent First we generate the i.i.d. innovation terms $\{\epsilon_t\}$ using the Laplace transform as discussed in \cite{Ridout2009}, then generate $\{Y_t\}$ from AR($1$) model defined in Eq. \eqref{GBmodel} with $\theta=0.3$. The plot of innovation terms $\{\epsilon_t\}$ for geometric tempered stable marginals, geometric gamma marginals and geometric inverse Gaussian are shown in Fig. \ref{fig4}. The time series $\{Y_t\}$ from AR($1$) model with $\theta=0.3$ for all the three cases are shown in Fig. \ref{fig3}.
	\begin{figure}[H]
		\centering
		\subfigure[Geometric tempered stable]{\includegraphics[width=0.33\textwidth]{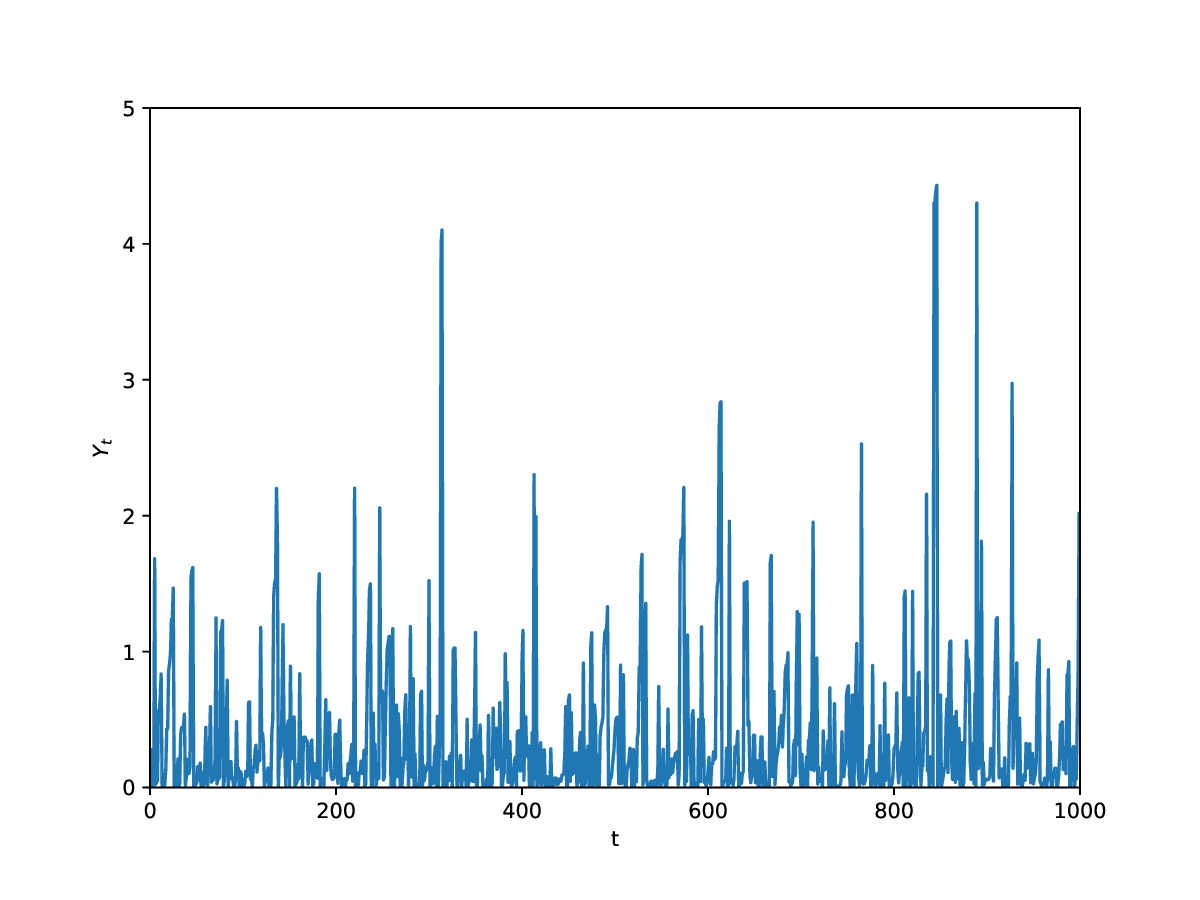}}\hfill
		\subfigure[Geometric gamma]{\includegraphics[width=0.33\textwidth]{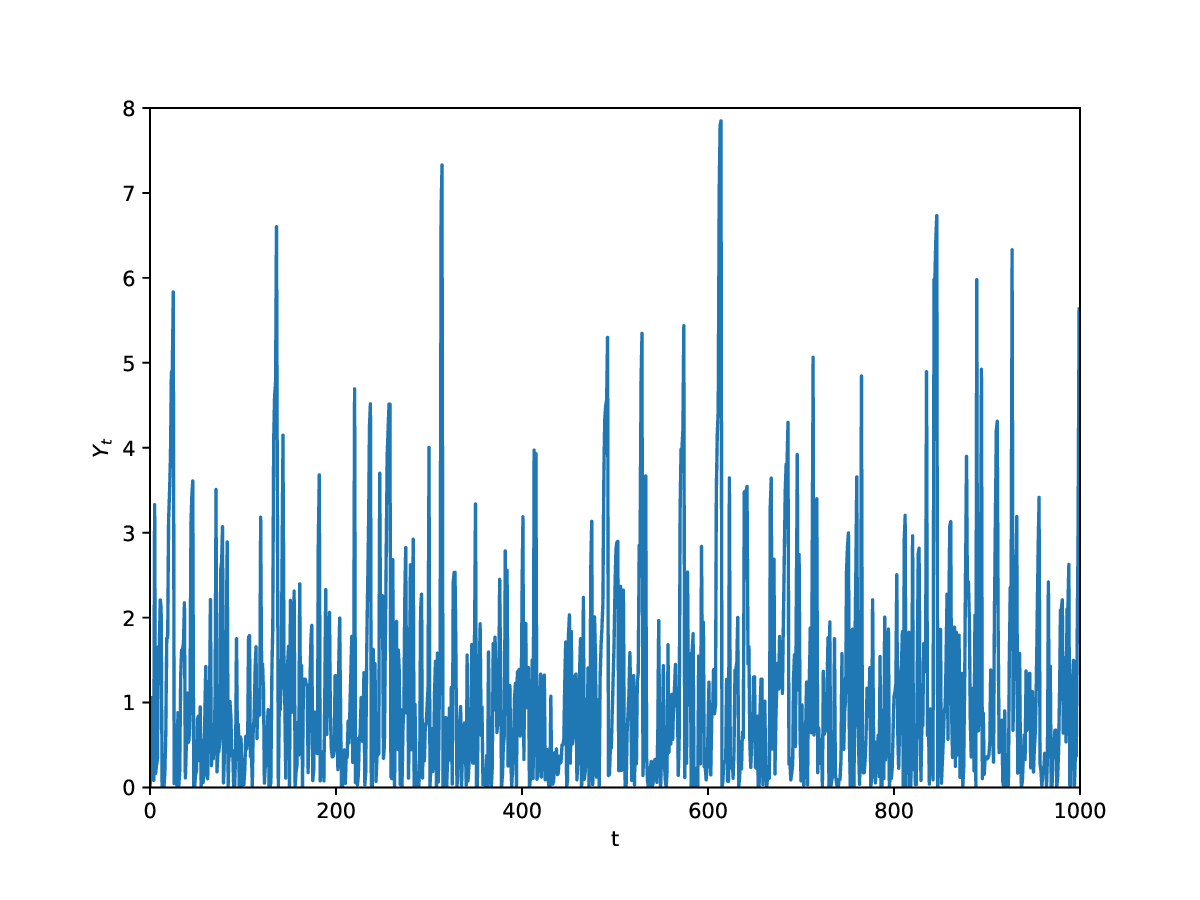}} \hfill
		\subfigure[Geometric inverse Gaussian]{\includegraphics[width=0.33\textwidth]{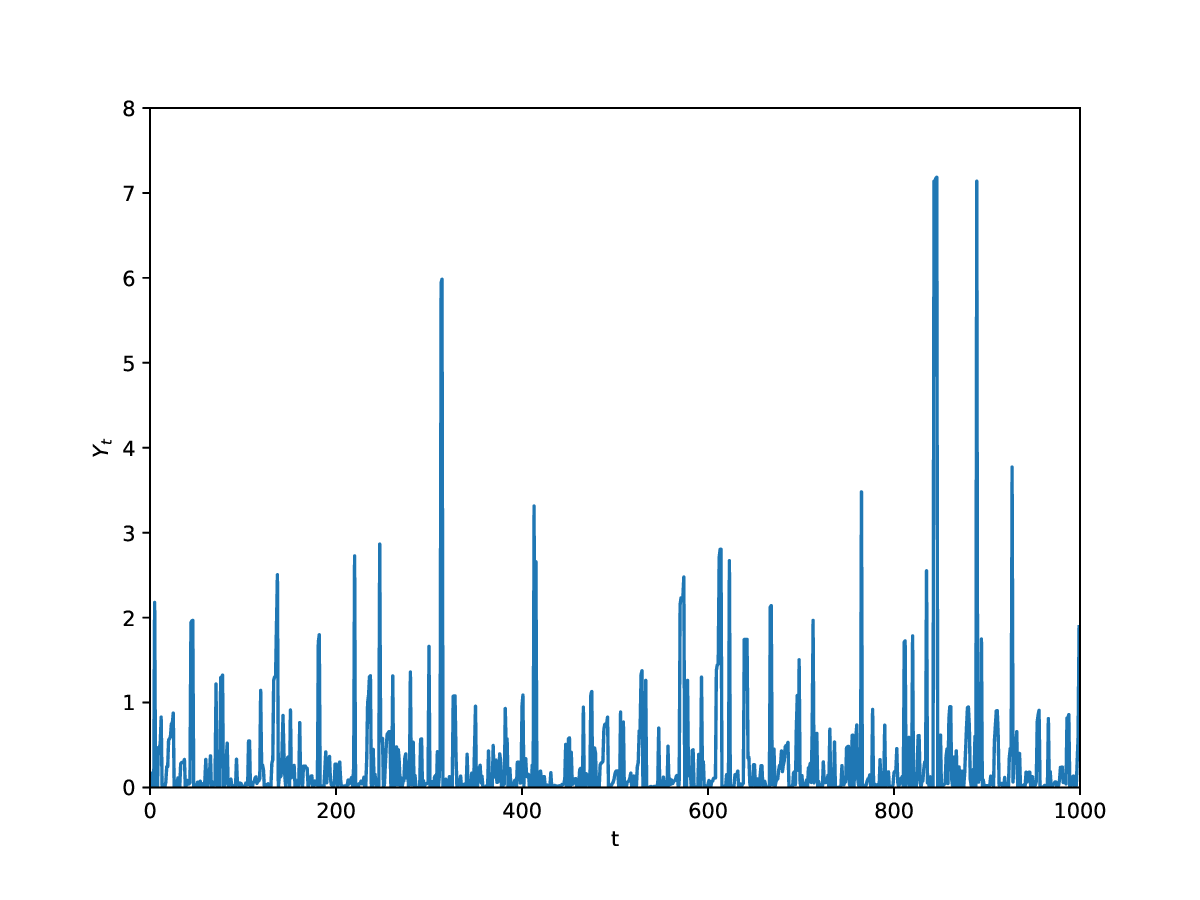}}
		\caption{The plot of time series $\{Y_t\}$ of length $1000$ from AR($1$) model with $\theta =0.3$ for (a) geometric tempered stable with parameters $\beta=0.6$ and $\lambda=1$ (b) geometric gamma with parameters $\beta=5$ and $\alpha=10$ (c) geometric inverse Gaussian with parameters $\gamma=1$ and $\delta=0.5$. }
		\label{fig3}
	\end{figure}
	In the next results, we discuss the form of the pdf of innovation terms using Laplace transform and complex inversion formula. For theory of complex inversion formula one can refer Chap. $4$ of \cite{Schiff1999}. 
	\begin{theorem}
		Consider a stationary AR$(1)$ process $\{Y_{n}\}$ as defined in Eq. \eqref{GBmodel}. If $\{Y_{n}\}$ is marginally distributed as geometric tempered stable with Laplace transform $\phi_{Y}(s)=\frac{1}{1+(s+\lambda)^\beta - \lambda^\beta}$, for $\beta=\frac{1}{m}, m=2,3,\cdots$, then the innovation terms $\{\epsilon_n\}$ also follow geometric tempered stable distribution with Laplace transform $\phi_{\epsilon}(s)=\frac{1}{1+\theta((s+\lambda)^\beta - \lambda^\beta)}$. Moreover, the pdf of innovation terms has the following integral form:
		\begin{equation}
			f_\epsilon(x)=\frac{e^{(s_0-\lambda)x}}{\theta \beta s_0^{\beta-1}}+\frac{1}{\pi} \int_{0}^{\infty} \frac{\theta e^{-x(y+\lambda)}y^{\beta}\sin(\pi\beta)}{1+2\theta(y^{\beta}\cos(\pi\beta)-\lambda^\beta)+\theta^2(y^{2\beta}-2\lambda^\beta y^\beta \cos(\pi\beta)+\lambda^{2\beta})} \,dy.\label{main_eq1}
		\end{equation}
		where $s_0=(\lambda^{\beta}-1/\theta)^{\beta-1}, \beta\in(0,1), \beta=\frac{1}{m}, m=2,3,\cdots \text{ and } \lambda>0.$
	\end{theorem}
	
	\begin{proof} We substitute $g(s)=\frac{1}{1+(s+\lambda)^\beta - \lambda^\beta}$ in Eq. \eqref{main_eq2} to obtain the Laplace transform of innovation terms $\phi_{\epsilon}(s)=\frac{1}{1+\theta((s+\lambda)^\beta - \lambda^\beta)}, \,\beta=\frac{1}{m},\, m=2,3,4,\hdots$. 
		
		\noindent Consider the function $G(s)= \frac{1}{1+\theta(s^\beta-\lambda^\beta)}.$ For the function $G(s)$, with  $\beta=\frac{1}{m},\, m=2,3,\hdots$, it follows that $s_0 = (\lambda^\beta-1/\theta)^{1/\beta}$ is a simple pole and $s_1=(0,0)$ is the branch point. We use complex inversion formula to compute the Laplace inverse of $G(s)$ which in turn gives the pdf $f_\epsilon(x)$ of innovation terms
		$$f_{\epsilon}(x) = \frac{1}{2\pi\iota}\int_{x_0-i\infty}^{x_0+i\infty}e^{sx}\phi_\epsilon(s)\,ds,$$ where $x_0>a$ is chosen such that the integrand is analytic for $Re(s)>a.$ 
		Consider the contour $ABCDEF$ in Fig. \ref{fig_1}  with branch point $s_1=(0,0)$, circular arcs $AB$ and $EF$ of radius $R$, arc $CD$ of radius $r$, line segments $BC$ and $DE$ parallel to $x-$axis and $AF$ is line segment from $x_0-iy$ to $x_0+iy$. For closed curve $C$ and poles $s_i$ inside $C$, Cauchy residue theorem states that, $$\frac{1}{2\pi\iota}\oint_C e^{sx}G(s)\,ds = \sum_{i=1}^{n}Res(e^{sx}G(s), s_i).$$ In limiting case, for contour $ABCDEF$, the integral on circular arcs $AB$ and $EF$ tend to $0$ as $R\to \infty.$ The integral over $CD$ also tends to $0$ as $r\to 0.$ We need to compute the following integral: 
		\begin{equation}
			\frac{1}{2\pi\iota}\int_{x_0-\iota \infty}^{x_0+\iota\infty} e^{sx}G(s)\,ds =Res(e^{sx}G(s), s_0) - \frac{1}{2\pi\iota}\int_{BC}e^{sx}G(s)\,ds - \frac{1}{2\pi\iota}\int_{DE}e^{sx}G(s)\,ds.\label{eq_2}
		\end{equation}
		
		
		\begin{figure}[h]
			\centering
			\includegraphics[width=0.7\textwidth]{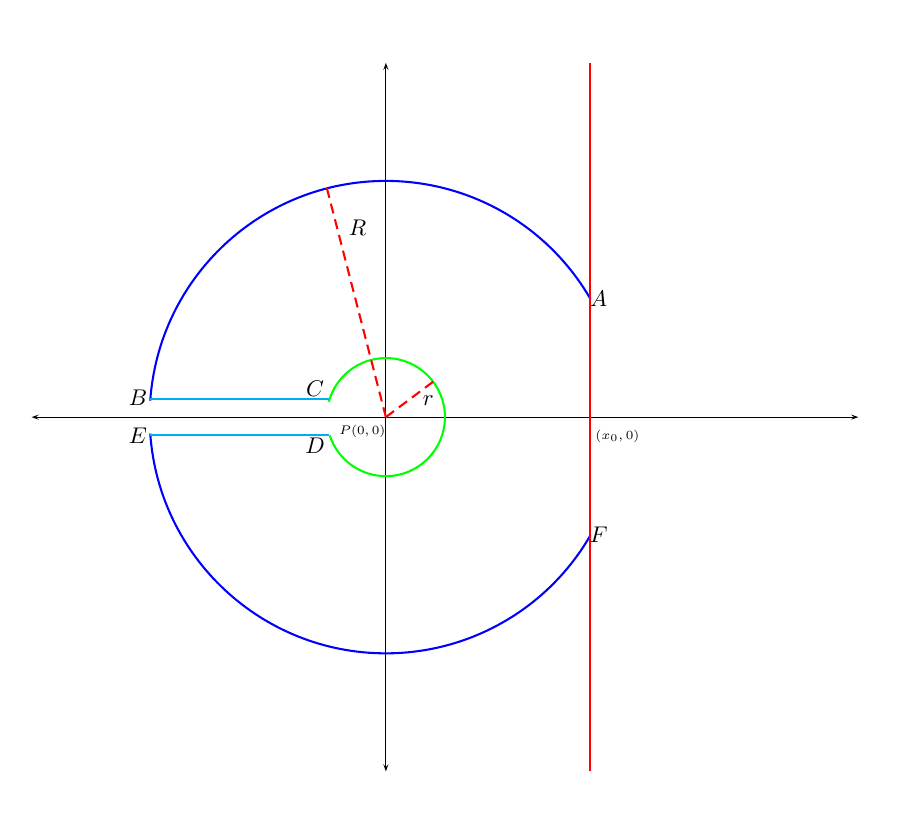}
			\caption{Contour plot with branch point at $P = (0,0)$ in anti-clockwise direction.}\label{fig_1}
		\end{figure}
		\noindent Now, along $BC$, let $s=ye^{\iota\pi}$, then $ds=-dy$ and,
		\begin{align}
			\frac{1}{2\pi\iota}\int_{-R}^{-r}\frac{e^{sx}}{1+\theta(s^\beta-\lambda^\beta)}\,ds &= \frac{1}{2\pi\iota}\int_{r}^{R}\frac{e^{-xy}}{1+\theta(y^{\beta}e^{\iota\pi\beta}-\lambda^{\beta})}\,dy \\
			&= \frac{1}{2\pi\iota}\int_{0}^{\infty}\frac{e^{-xy}}{1+\theta(y^{\beta}e^{\iota\pi\beta}-\lambda^{\beta})}\,dy.\label{eq_3}
		\end{align}
		Along $DE$, let $s=ye^{-\iota\pi}$, then $ds=-dy$ and,
		\begin{align}
			-\frac{1}{2\pi\iota}\int_{-r}^{-R}\frac{e^{sx}}{1+\theta(s^\beta-\lambda^\beta)}\,ds &= -\frac{1}{2\pi\iota}\int_{r}^{R}\frac{e^{-xy}}{1+\theta(y^{\beta}e^{-\iota\pi\beta}-\lambda^{\beta})}\,dy \\
			&= -\frac{1}{2\pi\iota}\int_{0}^{\infty}\frac{e^{-xy}}{1+\theta(y^{\beta}e^{-\iota\pi\beta}-\lambda^{\beta})}\,dy.\label{eq_4}
		\end{align}
		Now substitute Eq. \eqref{eq_3} and \eqref{eq_4} in Eq. \eqref{eq_2}.
		
		\begin{align*}
			\frac{1}{2\pi\iota}&\int_{x_0-\iota \infty}^{x_0+\iota\infty} e^{sx}G(s)\,ds =Res(e^{sx}G(s), s_0) - \frac{1}{2\pi\iota}\Big[\int_{0}^{\infty}\frac{e^{-xy}}{1+\theta(y^{\beta}e^{\iota\pi\beta}-\lambda^{\beta})}\,dy\\
			&-\frac{1}{2\pi\iota}\int_{0}^{\infty}\frac{e^{-xy}}{1+\theta(y^{\beta}e^{-\iota\pi\beta}-\lambda^{\beta})}\,dy\Big]\\
			&=Res(e^{sx}G(s), s_0) - \frac{1}{2\pi\iota}\int_{0}^{\infty}\frac{e^{-xy}(\theta y^\beta(e^{-\iota\pi\beta}-e^{\iota\pi\beta}))}{(1+\theta(y^{\beta}e^{\iota\pi\beta}-\lambda^{\beta}))(1+\theta(y^{\beta}e^{-\iota\pi\beta}-\lambda^{\beta}))}\,dy\\
			&=Res(e^{sx}G(s), s_0) + \frac{1}{\pi}\int_{0}^{\infty}\frac{e^{-xy}\theta y^\beta\sin(\pi\beta)}{1+2\theta(y^{\beta}\cos\pi\beta-\lambda^{\beta})+\theta^2(y^{2\beta}-2\lambda^{\beta}y^{\beta}\cos\pi\beta+\lambda^{2\beta})}\,dy
		\end{align*}
		Now we find the $Res(e^{sx}G(s))$ for poles at $s_0=(\lambda^\beta-1/\theta)^{m}, \, \beta =1/m,\; m=2,3,4,\cdots$. Also, note that all the poles $s_0<0$ are outside the analytic region, therefore we will calculate the residue at all those $s_0$ where $(\lambda^\beta-1/\theta)^m>0$ i.e. $\lambda>\frac{1}{\theta^m}$. We evaluate the residue as, $$\lim_{s \to s_0}\frac{(s-s_0)e^{sx}}{1+\theta(s^\beta-\lambda^\beta)}= \frac{e^{s_0x}}{\theta\beta(s_0)^{\beta-1}},$$ where $s_0=(\lambda^\beta-1/\theta)^{m}.$ The inverse Laplace of $G(s)$ is $$\mathcal{L}^{-1}\{G(s)\}=\frac{e^{s_0x}}{\theta\beta(s_0)^{\beta-1}}+\frac{1}{\pi}\int_{0}^{\infty}\frac{e^{-xy}\theta y^\beta\sin(\pi\beta)}{1+2\theta(y^{\beta}\cos\pi\beta-\lambda^{\beta})+\theta^2(y^{2\beta}-2\lambda^{\beta}y^{\beta}\cos\pi\beta+\lambda^{2\beta})}\,dy.$$
		We obtain the pdf by using the first translational property of inverse Laplace.
		\begin{align*}
			f_\epsilon(x) &= \mathcal{L}^{-1}\{G(s+\lambda)\}\, = e^{-\lambda x}\mathcal{L}^{-1}{G(s)}\\
			&= e^{-\lambda x}\Big[\frac{e^{s_0x}}{\theta\beta(s_0)^{\beta-1}}+\frac{1}{\pi}\int_{0}^{\infty}\frac{e^{-xy}\theta y^\beta\sin(\pi\beta)}{1+2\theta(y^{\beta}\cos\pi\beta-\lambda^{\beta})+\theta^2(y^{2\beta}-2\lambda^{\beta}y^{\beta}\cos\pi\beta+\lambda^{2\beta})}\,dy\Big].
		\end{align*}
		Hence, we obtain the form of pdf as mentioned in \eqref{main_eq1}.
	\end{proof}
	\begin{remark}
		For $\lambda=0, \phi_Y(s)=\frac{1}{1+s^{\beta}}$ and $\phi_{\epsilon}(s)=\frac{1}{1+\theta s^{\beta}}$, which is the Laplace transform of one side geometric stable distribution also known as Mittag-Leffler distribution. For $\theta=1$, poles will be $s_0=(-1)^m,\, m=2,3,4\cdots$. The residue for the pole corresponding to $s_0=1$ is $e^x.$ Then, the pdf becomes,
		\begin{align}
			f_{\epsilon}(x) = e^{x}+\frac{1}{\pi}\int_{0}^{\infty}\frac{e^{-xy}y^\beta\sin(\pi\beta)}{1+2 y^{\beta}\cos(\pi\beta) + y^{2\beta}}\,dy, \text{ where }s_0=1. \label{dist1}
		\end{align}
	\end{remark}
	
	\begin{remark}
		In \eqref{dist1}, we take $\beta=1/2$ or equivalently $m=2$ and obtain the density using the results from \cite{Abram}(pp. 303-304),
		\begin{align*}
			f_{\epsilon}(x) &= e^{x}+\frac{1}{\pi}\int_{0}^{\infty}\frac{e^{-xy}\sqrt{y}}{1+y}\,dy
			= e^{x} +\frac{1}{\sqrt{\pi x}}-e^x+e^x Erf(\sqrt{x})\\
			& =\frac{1}{\sqrt{\pi x}}+e^x Erf(\sqrt{x}).
		\end{align*}
		As $x\to 0,\, Erf(\sqrt{x}) \to 0 \text{ and }f_{\epsilon}(x) \sim \frac{1}{\sqrt{\pi x}},$ which match with  the asymptotic behaviour discussed in Theorem \ref{Theorem2.1}.
	\end{remark}
	
	\begin{proposition}
		For a stationary AR$(1)$ process with $\{Y_{n}\}$ defined as in Eq. \eqref{GBmodel} with marginals distributed as geometric gamma with Laplace transform $\phi_{Y}(s)=\frac{1}{1+\alpha\log(\frac{s+\beta}{\beta})}$, then the innovation terms $\{\epsilon_n\}$ also follow geometric gamma with Laplace transform $\phi_{\epsilon}(s)=\frac{1}{1+\alpha\theta\log(\frac{s+\beta}{\beta})}$. Moreover, the pdf of innovation terms has the following integral form: 
		$$f_{\epsilon}(x)=\frac{e^{s_0 x}(s_0+\beta)^{\alpha\theta}}{\beta^{\alpha\theta}} -e^{-\beta x} \int_{0}^{\infty}\frac{\alpha\theta e^{-xy}}{1+2\alpha\theta \log(y/\beta)+\pi^2\alpha^2\theta^2 +\alpha^2\theta^2(\log(y/\beta))^2}\,dy,$$ where $s_0=\beta(e^{-1/\alpha\theta}-1),\, \alpha>0 \text{ and } \beta>0.$
	\end{proposition}
	\begin{proof}
		Again we substitute $g(s)=\frac{1}{1+\alpha\log(\frac{s+\beta}{\beta})}$ in Eq. \eqref{main_eq2} and get the Laplace transform of innovation terms $\phi_{\epsilon}(s)=\frac{1}{1+\alpha\theta\log(\frac{s+\beta}{\beta})}$. Again we use the complex inversion formula to obtain the pdf of innovation terms. We consider the function $G(s)=\frac{1}{1+\theta\alpha\log(\frac{s+\beta}{\beta})}$ and the pole for $G(s)$ at $s_0=\beta(e^{-1/\alpha\theta}-1)$ and branch point at $s = -\beta$.  The residue corresponding to $s_0$ is $\frac{e^{s_0x}(s_0+\beta)^{\alpha\theta}}{\beta^{\theta\alpha}}.$ Using a similar contour as given in Fig. \ref{fig_1} and the same steps to obtain the pdf as $$f_{\epsilon}(x)=\frac{e^{s_0 x}(s_0+\beta)^{\alpha\theta}}{\beta^{\alpha\theta}} -e^{-\beta x} \int_{0}^{\infty}\frac{\alpha\theta e^{-xy}}{1+2\alpha\theta \log(y/\beta)+\pi^2\alpha^2\theta^2 +\alpha^2\theta^2(\log(y/\beta))^2}\,dy.$$
	\end{proof}
	
	\begin{proposition}
		For a stationary AR$(1)$ process with $\{Y_{n}\}$ defined as in Eq. \eqref{GBmodel} with marginals distributed as geometric inverse Gaussian with Laplace transform $\phi_{Y}(s)=\frac{1}{1+\gamma\delta\Bigl\{\sqrt{1+\frac{2s}{\gamma^2}}-1\Bigr\}}$, then the innovation terms $\{\epsilon_n\}$ also follows geometric inverse Gaussian with Laplace transform $\phi_{\epsilon}(s)=\frac{1}{1+\theta\gamma\delta\Bigl\{\sqrt{1+\frac{2s}{\gamma^2}}-1\Bigr\}}$. Moreover, the pdf of innovation terms has the following integral form: 
		$$f_{\epsilon}(x)=\frac{e^{s_0 x}\sqrt{2s_0+\gamma^2}}{\theta\delta} +\frac{\gamma^2e^{-x}}{2\pi}\int_{0}^{\infty}\frac{\theta\delta\sqrt{y}e^{-\frac{2xy}{\gamma^2}}}{1-2\theta\delta\gamma+\theta^2\delta^2(y+\gamma^2-2\gamma\sqrt{y}\cos{\theta})}\,dy,$$ where $s_0=\frac{\gamma^2}{2}(1-\frac{1}{\theta\gamma\delta})^2-\frac{\gamma^2}{2}, \,\delta>0 \text{ and } \gamma>0.$
	\end{proposition}
	\begin{proof}
		The Laplace transform of innovation terms $\phi_{\epsilon}(s)=\frac{1}{1+\theta\gamma\delta\Bigl\{\sqrt{1+\frac{2s}{\gamma^2}}-1\Bigr\}}$ is straight forward from Eq. \eqref{main_eq2}. Now we use the complex inversion formula to obtain the pdf of innovation terms as done in previous theorem. We consider the function $G(s)=\frac{1}{1+\theta\delta\gamma\Bigl\{\sqrt{1+\frac{2s}{\gamma^2}}-1\Bigr\}}$ and the pole for $G(s)$ is $s_0=\frac{\gamma^2}{2}\Bigl\{1-\frac{1}{\theta\delta\gamma}\Bigr\} - \frac{\gamma^2}{2}.$ The residue corresponding to $s_0$ is $\frac{e^{s_0x}\sqrt{2s_0+\gamma^2}}{\theta\delta}.$ We use the same steps to obtain the pdf as $$f_{\epsilon}(x)=\frac{e^{s_0 x}\sqrt{2s_0+\gamma^2}}{\theta\delta} +\frac{\gamma^2e^{-x}}{2\pi}\int_{0}^{\infty}\frac{\theta\delta\sqrt{y}e^{-\frac{2xy}{\gamma^2}}}{1-2\theta\delta\gamma+\theta^2\delta^2(y+\gamma^2-2\gamma\sqrt{y}\cos{\theta})}\,dy.$$
	\end{proof}
	
	\noindent We know that for a random process, the moments (if exists) uniquely describe the properties of the random variable $X$. We provide the $k^{th}$ order moments using the Laplace transform in the following result.
	The $k^{th}$ order moment of random variable $X$ is evaluated as:
	$\mathbb{E}(X^{k}) = (-1)^{k}\phi^{(k)}(s)$ for $s=0$ and $k\in \mathbb{N}$.
	
	\begin{enumerate}[(a)]
		\item Geometric tempered stable innovations: $\phi_\epsilon(s) = \dfrac{1}{1+\theta\{(s+\lambda)^\beta-\lambda^\beta\}}$. Then, \\
		$\mathbb{E}(\epsilon) = \theta\beta\lambda^{\beta-1},\quad
		\mathbb{E}(\epsilon^2) = \theta\beta(\beta-1)\lambda^{\beta-2}-2\theta^2\beta^2(\lambda)^{2\beta-2},\lambda >0.$
		
		\item Geometric gamma innovations: $\phi_{\epsilon}(s)=\dfrac{1}{1+\alpha\theta\log\left(\dfrac{s+\beta}{\beta}\right)}$. Then, \\
		$\mathbb{E}(\epsilon) = \dfrac{\theta\alpha}{\beta}, \quad
		\mathbb{E}(\epsilon^2) = \dfrac{\theta\alpha + 2\theta^{2}\alpha^2}{\beta^2}.$
		
		\item Geometric inverse Gaussian innovations: $\phi_\epsilon(s) = \dfrac{1}{1+\theta\gamma\delta\Bigl\{\sqrt{1+\frac{2s}{\gamma^2}}-1\Bigr\}}$. Then, \\
		$\mathbb{E}(\epsilon) = \dfrac{\theta\delta}{\gamma}, \quad
		\mathbb{E}(\epsilon^2) = \dfrac{2\theta^{2}\delta^{2}}{\gamma^2} + \dfrac{\theta\delta}{\gamma^3}.$
	\end{enumerate}

	
	\subsection{Generalisation to $k^{th}$ order autoregressive processes}
	We define the generalised form of the process developed in previous section by Eq. \eqref{GBmodel} as follows: 
	
	\begin{equation} 
		Y_{n} =
		\begin{cases}
			\epsilon_n, & \text{ with probability } \theta,\\
			Y_{n-1} + \epsilon_n, & \text{ with probability } 1-\theta_1,\\
			Y_{n-2} + \epsilon_n, & \text{ with probability } 1-\theta_2,\\
			\vdots\\
			Y_{n-k} + \epsilon_n, & \text{ with probability } 1-\theta_k,   
		\end{cases} \label{GGBmodel},
	\end{equation}
	where $\sum_{i=1}^{k}\theta_i = 1-\theta, \,$  $0\leq\theta \leq1,\, i=1,2, \hdots, k.$ Also, $\{\epsilon_n\}$ are independent of $\{Y_{n-1}, Y_{n-2}, \cdots, \}$. We write the Laplace transform $\phi_{Y_n}(s)$ for the model defined in Eq. \eqref{GGBmodel}. 
	\begin{align*}
		\phi_{Y_n}(s) &= \theta\phi_{\epsilon_n}(s) + \theta_1\phi_{Y_{n-1}}(s)\phi_{\epsilon_n}(s) + \cdots + \theta_k\phi_{Y_{n-k}}(s)\phi_{\epsilon_n}(s)\\
		&= \phi_{\epsilon_n}(s)\{\theta + \theta_1\phi_{Y_{n-1}(s)} + \cdots + \theta_k\phi_{Y_{n-k}(s)}\}.
	\end{align*}
	\noindent Since series is stationary, we get,
	\begin{align*}
		\phi_{Y}(s)&=\phi_{\epsilon}(s){\theta + \sum_{i=1}^{k}\theta_i\phi_{Y}(s)} =\phi_{\epsilon}(s)\{\theta + {(1-\theta)}\phi_{Y}(s)\};\\
		\phi_{\epsilon}(s)&=\frac{\phi_Y(s)}{\theta+(1-\theta)\phi_Y(s)}.
	\end{align*}
	Hence, we obtain the similar form for innovation terms $\{\epsilon_{n}\}$ as obtained in previous section.
	Next we define AR$(1)$ process with $|\theta|<1$ that is, stationary process and find the Laplace transform of the innovation terms.
	\begin{proposition}\label{ARmodel}
		Consider the AR$(1)$ process $Y_n=\theta Y_{n-1}+ \epsilon_n, \, |\theta|<1$ is strictly stationary with Laplace transform of marginals as $\phi_{Y}(s)=\frac{1}{1+g(s)}$ then the Laplace transform of innovation terms $\{\epsilon_n\}$ is $\phi_{\epsilon}(s)=\frac{1+g(\theta s)}{1+g(s)}$. 
	\end{proposition} 
	\begin{proof}
		We have stationary AR$(1)$ process with $|\theta|<1$ defined as 
		\begin{align*}
			Y_n &= \theta Y_{n-1} +\epsilon_n, \\
			\phi_{Y_n}(s) &= \phi_{Y_{n-1}}(\theta s)\phi_{\epsilon}(s)\\ 
			\phi_{\epsilon}(s) &= \frac{\phi_{Y}(s)}{\phi_{Y}(\theta s)}\,
			=\frac{1+g(\theta s)}{1+g(s)}.
		\end{align*}
	\end{proof}
	\begin{enumerate}[(a)]
		\item Geometric tempered stable:
		For $g(s) = (s+\lambda)^{\beta}-\lambda^{\beta},$ then $ \phi_\epsilon(s) = \frac{1+(\theta s+\lambda)^\beta-\lambda^\beta}{1 +(s+\lambda)^{\beta}-\lambda^{\beta}}, \, \lambda>0, \text{ and }\beta\in(0,1).$ 
		
		\item Geometric gamma:
		For $g(s) = \alpha\log(1+\frac{s}{\beta}),$ then $\phi_\epsilon(s) = \frac{1+\alpha\log\Bigl(\frac{\theta s+\beta}{\beta}\Bigr)}{1+\alpha\log\Bigl(\frac{s+\beta}{\beta}\Bigr)}, \, \alpha>0 \text{ and } \beta>0.$
		
		\item Geometric inverse Gaussian:
		For $g(s) = \delta\gamma\Bigl\{\sqrt{1+\frac{2s}{\gamma^2}}-1\Bigr\}$, then $\phi_\epsilon(s) = \frac{1+\delta\gamma\Bigl\{\sqrt{1+\frac{2\theta s}{\gamma^2}}-1\Bigr\}}{1+\delta\gamma\Bigl\{\sqrt{1+\frac{2s}{\gamma^2}}-1\Bigr\}}, \, \delta>0 \text{ and } \gamma>0.$
	\end{enumerate}
	
	\section{Parameter estimation and simulation}\label{section4}
	In this section, we estimate the parameters of the model defined in \ref{ARmodel} using conditional least square (CLS) and method of moments (MOM). We first apply the CLS method to estimate the parameter $\theta$ and then use the MOM for the parameters of marginals in next subsection. 
	
	\subsection{Estimation by conditional least square and method of moments}
	\noindent The conditional likelihood function is given by,
	$$L(\theta, \lambda,\beta)=\sum_{t=1}^{n}(Y_t-\mathbb{E}(Y_t|Y_{t-1}))^2,\text{ where } \mathbb{E}[Y_t|Y_{t-1}] = \theta Y_{t-1} + \mathbb{E}(\epsilon_t).$$
	\begin{enumerate}[(a)]
		\item Geometric tempered stable: For the model defined in \ref{ARmodel}, first we assume that innovations are from distribution with Laplace transform $\phi_\epsilon(s) = \frac{1+(\theta s+\lambda)^\beta-\lambda^\beta}{1 +(s+\lambda)^{\beta}-\lambda^{\beta}}, \, \lambda>0 \text{ and }\beta\in(0,1).$ The first order and second order theoretical moments of innovation terms \{$\epsilon_t\}$ are approximated by the empirical moments, which are given by
		\begin{align} \label{MOM1}
			\hat{m}_1 &= \sum_{t=1}^{n}\dfrac{\epsilon_t}{n}=(1-\theta)\beta\lambda^{\beta-1}, \\
			\hat{m}_2 &= \sum_{t=1}^{n}\dfrac{\epsilon_t^2}{n} = \beta(\beta-1)(\theta^2-1)\lambda^{\beta-2} - 2\beta^2\lambda^{2\beta-2}(\theta-1).
		\end{align}
		
		We substitute $\hat{m}_1$ from \eqref{MOM1} to likelihood function $L(\theta, \lambda, \beta)$ and obtain, 
		$$
		L = \sum_{t=1}^{n}(Y_t-\theta Y_{t-1} - (1-\theta)\beta\lambda^{\beta-1})^2.
		$$
		Take derivative with respect to unknown parameters $\theta, \lambda$ and $\beta$ and obtain the following relation,
		\begin{align}
			\dfrac{\partial{L}}{\partial{\lambda}} &= \sum_{t=1}^{n}2(Y_t-\theta Y_{t-1} - (1-\theta)\beta\lambda^{\beta-1})((1-\theta)\beta\lambda^{\beta-1}),\\
			\dfrac{\partial{L}}{\partial{\beta}} &=\sum_{t=1}^{n}2(Y_t-\theta Y_{t-1} - (1-\theta)\beta\lambda^{\beta-1})((1-\theta)\lambda^{\beta-1}+(1-\theta)\beta\lambda^{\beta-1}\log{(\beta-1)}),\\
			\dfrac{\partial{L}}{\partial{\theta}} &= \sum_{t=1}^{n}2(Y_t-\theta Y_{t-1} - (1-\theta)\beta\lambda^{\beta-1})(\beta\lambda^{\beta-1}- Y_{t-1}).
		\end{align}
		Solving the above equations, we get the estimate for $\theta$,
		$$\hat{\theta} = \dfrac{\sum_{t=1}^{n}Y_tY_{t-1}-n\bar{Y}_{t-1}^2}{\sum_{t=1}^{n}(Y_{t-1}-\bar{Y}_{t-1})^2},$$
		where $\bar{Y}_{t-1} = \dfrac{\sum_{t=1}^{n}Y_{t-1}}{n}. $
		\noindent Now we use first and second order moments to estimate the remaining parameters $\lambda$ and $\beta$.
		\begin{align*}
			\hat{m}_1 &= (1-\hat{\theta})\beta\lambda^{\beta-1}, \\
			\hat{m}_2 &=  \beta(\beta-1)(\hat{\theta}^2-1)\lambda^{\beta-2} - 2\beta^2\lambda^{2\beta-2}(\hat{\theta}-1).
		\end{align*}
		After rearranging the terms we get the following non-linear relation between $\lambda$ and $\beta$, $$\beta = 1 - \lambda\dfrac{\hat{m}_2(1-\hat{\theta})-2\hat{m}_1^2}{\hat{m}_1(1-\hat{\theta}^2)}.$$ To solve it further we use \textit{fsolve()} function available in python scipy package.
		Also, note that the estimate for $\theta$ using CLS will be same for all the cases.
		\item Geometric gamma: For this case the Laplace transform of innovation terms are \\$\phi_\epsilon(s) = \dfrac{1+(\log(\theta s+\beta/\beta))^\alpha}{1 +(\log(s+\beta/\beta))^{\alpha}}, \, \alpha>0 \text{ and }\beta>0.$ The first and second order moments of innovation terms will be, 
		\begin{align*} \label{MOM2}
			\hat{m}_1 &= (1-\hat{\theta})\dfrac{\alpha}{\beta}, \\
			\hat{m}_2 &= \dfrac{\alpha(1-\hat{\theta}^2)}{\beta^2} + \dfrac{2\alpha^2(1-\hat{\theta})}{\beta^2}.
		\end{align*}
		Using these moments we get the the estimates as $\hat{\beta} = \dfrac{\hat{m}_1(1-\hat{\theta}^2)}{\hat{m}_2(1-\hat{\theta}-2\hat{m}_1^2)}$ and $\hat{\alpha} = \dfrac{\hat{m}_1\hat{\beta}}{1-\hat{\theta}}.$
		
		\item Geometric inverse Gaussian: The Laplace transform of innovation terms are \\$\phi_\epsilon(s) = \dfrac{1+\delta(\sqrt{2\theta s+\gamma^2}-\gamma)}{1+\delta(\sqrt{2s+\gamma^2}-\gamma)}, \, \gamma>0\text{ and }\delta>0.$ The first and second order moments of innovation terms will be, 
		\begin{align*}
			\hat{m}_1 &= (1-\hat{\theta})\dfrac{\delta}{\gamma}, \\
			\hat{m}_2 &= 2(1-\hat{\theta})\dfrac{\delta^2}{\gamma^2} + (1-\hat{\theta}^2)\dfrac{\delta}{\gamma^3}.
		\end{align*}
		Using these moments we get the the estimates as $\hat{\gamma} = \sqrt{\dfrac{\hat{m}_1(1-\hat{\theta}^2)}{\hat{m}_2(1-\hat{\theta})-2\hat{m}_1^2}}$ and $\hat{\delta} = \dfrac{\hat{m}_1\hat{\gamma}}{1-\hat{\theta}}.$
	\end{enumerate}
	
	\subsection{Simulation}
	We use simulation to further assess the performance of estimation method. We simulate the $500$ trajectories each of length $1000$ for geometric tempered and geometric gamma case. We use the method of Laplace transform to simulate the innovation terms $\{\epsilon_n\}$ as described in \cite{Ridout2009} and then generate the time series $\{Y_n\}$ from AR($1$) model. The true model parameter is $\theta = 0.3$ and distribution parameters for geometric tempered stable, geometric gamma and geometric inverse Gaussian are tabulated in \ref{tab1}. The boxplots for two cases are shown in Fig. \ref{fig5}. 
	\begin{table}[H]
		\caption{Estimation of parameters using CLS and MOM for geometric tempered stable and geometric gamma case.} \label{tab1}
		\centering
		\begin{tabular}{ P{2.0cm}P{2cm}P{2cm}P{2cm}}
			\hline
			&&&
			\multirow{2}{4em}{}\\
			& parameter $1$ & parameter $2$ & parameter $3$ \\               
			\hline
			&&&
			\multirow{2}{4em}{}\\
			True values & $\beta = 0.6 $ & $\lambda = 1 $ & 
			$\rho = 0.3$\\    
			Est. values &  $\hat{\beta} = 0.206$ & $\hat{\lambda} = 0.917$ & $\hat{\rho} = 0.297$\\
			&&&
			\multirow{2}{4em}{}\\
			True values &$\beta = 1 $ & $\alpha = 2 $ & $\rho = 0.3$ \\    
			Est. values & $\hat{\beta} = 1.189$ & $\hat{\alpha} = 2.422$ & $\hat{\rho} = 0.298$ \\
			&&&
			\multirow{2}{4em}{}\\
			True values &$\gamma = 1 $ & $\delta = 2 $ & $\rho = 0.3$ \\    
			Est. values & $\hat{\gamma} = 1.028$ & $\hat{\delta} = 2.054$ & $\hat{\rho} = 0.296$ \\
			\hline
		\end{tabular}
	\end{table}
	\noindent From the data in table and boxplots we observe that the CLS method gives good estimate for parameter $\rho$, whereas the estimation of other parameters by MOM have variance. The estimated value of $\beta$ from geometric tempered stable is not good. The relation of $\beta$ and $\lambda$ is non-linear therefore we solved it by using \textit{fsolve} function defined in python's scipy library which is based on numerical methods. 
	\begin{figure}[H]
		\centering
		\subfigure[Geometric tempered stable]{\includegraphics[width=0.33\textwidth]{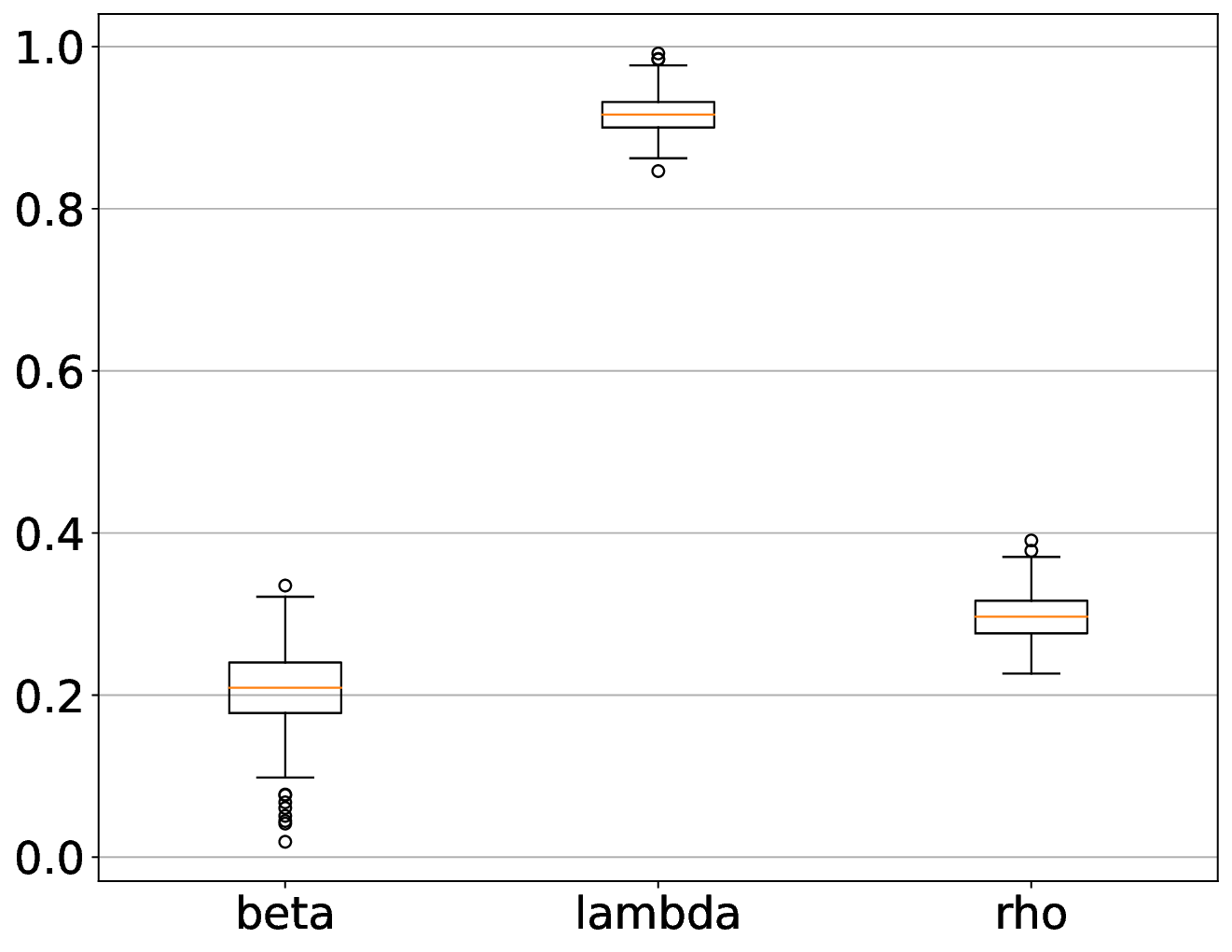}}\hfill
		\subfigure[Geometric gamma]{\includegraphics[width=0.33\textwidth]{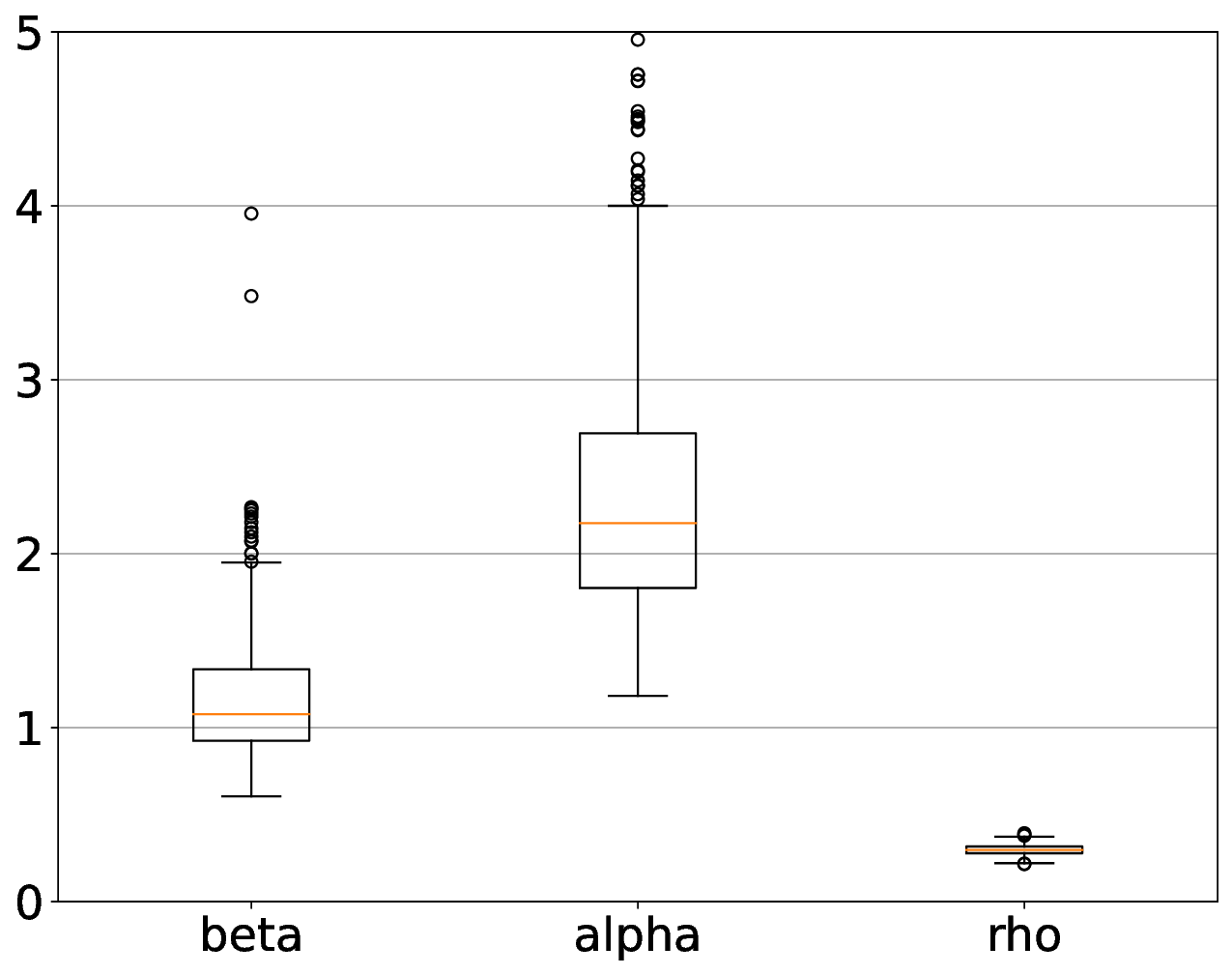}}
		\subfigure[Geometric inverse Gaussian]{\includegraphics[width=0.33\textwidth]{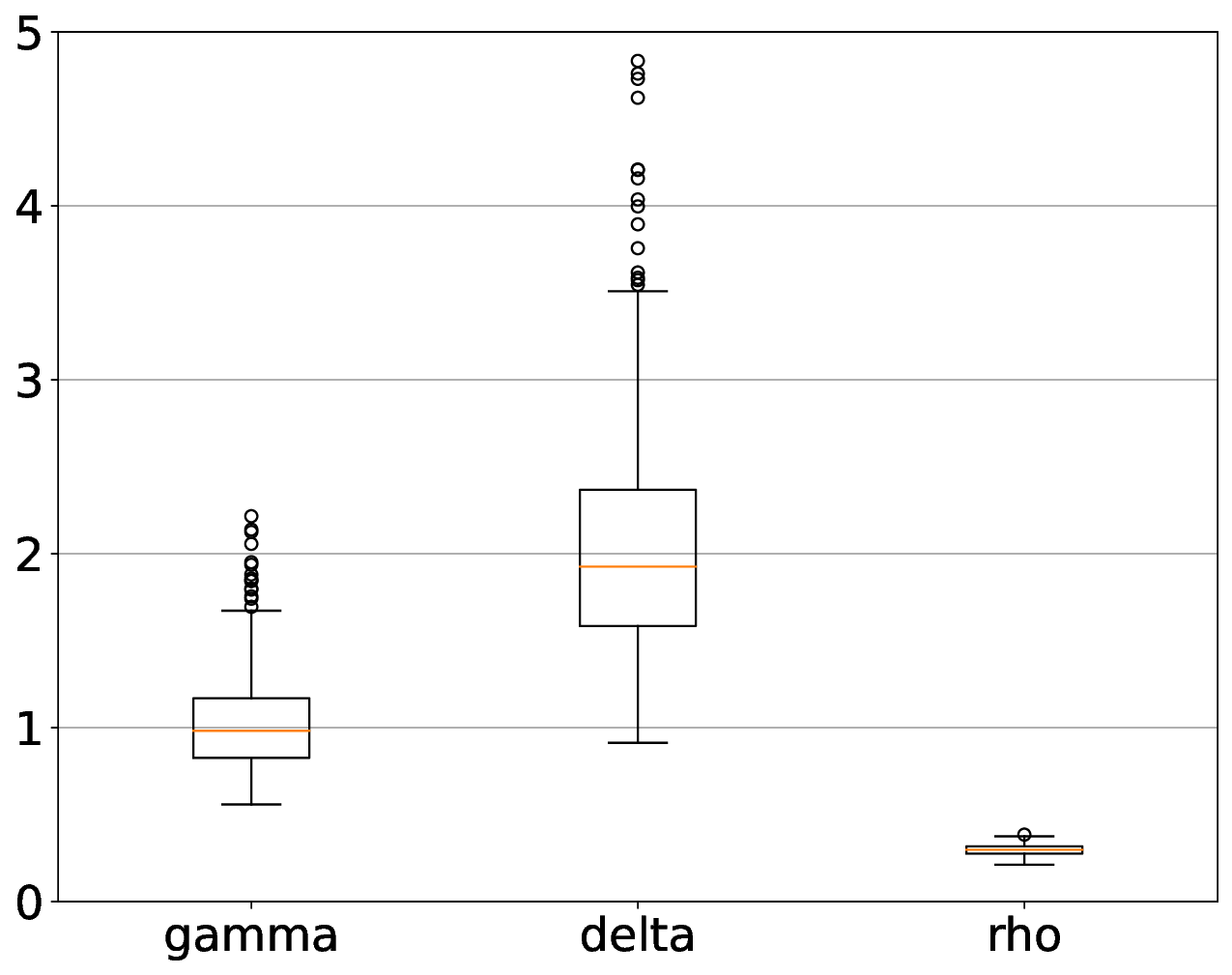}}
		\caption{The boxplots for model parameter estimates from AR($1$) model with true value of $\theta =0.3$ for (a) geometric tempered stable with true parameter values of $\beta=0.6$ and $\lambda=1$ (b) geometric gamma with true parameter values of $\beta=1$ and $\alpha=2$ (c) geometric inverse Gaussian with parameters $\gamma=1$ and $\delta=2.$
		}\label{fig5}
	\end{figure}
	
	\section{Conclusion}\label{conclusion}
	In this paper we use the Bernstein function $g(s)$ which is the Laplace exponent of a positive infinitely divisible random variable to define gid random variables with Laplace transform of the form $\dfrac{1}{1+g(s)}$. We also find the Laplace transform of mixtures of some particular gid random variables which is a new class of marginals to study. A new autoregressive process of order 1 with gid distribution is considered. We deduce that if marginals of AR$(1)$ defined in Eq. \eqref{GBmodel} are gid then the innovation terms are also gid with scale change of $\theta$. We find the integral form of the pdf of innovation terms using the Laplace transform and complex inversion method for three cases namely, geometric tempered stable, geometric gamma and geometric inverse Gaussian subordinators. Further, moments play an important role to study the characteristics of pdf. We have calculated the first and second order moments for these three gid random varaiables. Next we generalised the AR process to $k^{th}$ order and also proposed AR($1$) model defined in \ref{ARmodel} with marginals having Laplace transform of the form $\dfrac{1}{1+g(s)}$. At last, we have estimated the parameters of the model defined in \ref{ARmodel} using CLS and MOM and simulation study implies that the estimates are satisfactory. \\
	
	\noindent{\bf Acknowledgements:} Monika S. Dhull would like to thank the Ministry of Education (MoE), India for supporting her PhD research. Further, Arun Kumar would like to express his gratitude to Science and Engineering Research Board (SERB), India for financial support under the MATRICS research grant MTR/2019/000286. \\
	
	\noindent {\bf Data Availability:} No real world data is used for research described in this article.
	\bibliography{bibfile}
	
\end{document}